\documentclass{amsart}
\usepackage{amssymb}
\usepackage{amsmath}
\usepackage{amsfonts}

\newtheorem{theorem}{Theorem}
\theoremstyle{plain}

\newtheorem{corollary}{Corollary}

\newtheorem{definition}{Definition}
\newtheorem{example}{Example}

\newtheorem{lemma}{Lemma}

\newtheorem{remark}{Remark}

\numberwithin{equation}{section}
\input{tcilatex}

\begin{document}
\title[]{Characterizations of right modular groupoids by $\left( \in ,\in
\vee q_{k}\right) $-fuzzy ideals}
\author{}
\maketitle

\begin{center}
\textbf{Madad Khan and Shamas-ur-Rehman}

\ Department of Mathematics

COMSATS Institute of Information Technology

\ Abbottabad, Pakistan.

\textit{E-mail:} madadmath@yahoo.com

\textit{E-mail: }shamas\_200814@yahoo.com

\bigskip
\end{center}

\textbf{Abstract. }In this paper, we have introduced the concept of $\left(
\in ,\in \vee q\right) $-fuzzy ideals in a right modular groupoid. We have
discussed several important features of a completely regular right modular
groupoid by using the $\left( \in ,\in \vee q\right) $-fuzzy left (right,
two-sided) ideals, $\left( \in ,\in \vee q\right) $-fuzzy (generalized)
bi-ideals and $\left( \in ,\in \vee q\right) $-fuzzy $(1,2)$-ideals. We have
also used the concept of $\left( \in ,\in \vee q_{k}\right) $-fuzzy left
(right, two-sided) ideals, $\left( \in ,\in \vee q_{k}\right) $-fuzzy
quasi-ideals $\left( \in ,\in \vee q_{k}\right) $-fuzzy bi-ideals and $%
\left( \in ,\in \vee q_{k}\right) $-fuzzy interior ideals in completely
regular right modular groupoid and proved that the $\left( \in ,\in \vee
q_{k}\right) $-fuzzy left (right, two-sided), $\left( \in ,\in \vee
q_{k}\right) $-fuzzy (generalized) bi-ideals, and $\left( \in ,\in \vee
q_{k}\right) $-fuzzy interior ideals coincide in a completely regular right
modular groupoid.


\begin{center}
\bigskip

{\LARGE Introduction}
\end{center}

The fundamental concept of fuzzy sets was first introduced by Zadeh $[18]$
in $1965$. Given a set $X$, a fuzzy subset of $X$ is, by definition an
arbitrary mapping $f:X\rightarrow \lbrack 0,1]$ where $[0,1]$ is the unit
interval. Rosenfeld introduced the definition of a fuzzy subgroup of a group 
$[15]$. Kuroki initiated the theory of fuzzy bi ideals in semigroups $[8]$.
The thought of belongingness of a fuzzy point to a fuzzy subset under a
natural equivalence on a fuzzy subset was defined by Murali $[11]$. The
concept of quasi-coincidence of a fuzzy point to a fuzzy set was introduce
in $[14]$. Jun and Song \ introduced $\left( \alpha ,\beta \right) $-fuzzy
interior ideals in semigroups $[4]$.

In this paper we have characterized non-associative algebraic structures
called right modular groupoids by their $\left( \in ,\in \vee q_{k}\right) $%
-fuzzy ideals. A right modular groupoid $M$ is non-associative and
non-commutative algebraic structure mid way between a groupoid and a
commutative semigroup.

The concept of a left almost semigroup (LA-semigroup) $[5]$ or a right
modular groupoid was first given by M. A. Kazim and M. Naseeruddin in $1972$%
. A right modular groupoid $M$ is a groupoid having the left invertive law,%
\begin{equation}
(ab)c=(cb)a\text{, for\ all }a\text{, }b\text{, }c\in M\text{.}  \tag{$1$}
\end{equation}%
In a right modular groupoid $M$, the following medial law $[5]$ holds,%
\begin{equation}
(ab)(cd)=(ac)(bd)\text{, for\ all }a\text{, }b\text{, }c\text{, }d\in M\text{%
.}  \tag{$2$}
\end{equation}%
The left identity in a right modular groupoid if exists is unique $[12]$. In
a right modular groupoid $M$ with left identity the following paramedial law
holds [13],%
\begin{equation}
(ab)(cd)=(dc)(ba),\text{ for\ all }a,b,c,d\in M.  \tag{$3$}
\end{equation}%
If a right modular groupoid $M$ contains a left identity, then,

\begin{equation}
a(bc)=b(ac)\text{, for\ all }a\text{, }b\text{, }c\in M\text{.}  \tag{$4$}
\end{equation}

\bigskip

{\huge Preliminaries}

\bigskip

Let $M$ be a right modular groupoid, by a subgroupoid of $M,$ we means a
non-empty subset $A$ of $M$ such that $A^{2}\subseteq A$. A non-empty subset 
$A$ of a right modular groupoid $M$ is called left (right) ideal of $M$ if $%
MA\subseteq A$ $(AM\subseteq A)$. $A$ is called two-sided ideal or simply
ideal if it is both a left and a right ideal of $M$. A non empty subset $A$
of a right modular groupoid $M$ is called generalized bi-ideal of $M$ if $%
(AM)A\subseteq A$. A subgroupoid $A$ of $M$ is called bi-ideal of $M$ if $%
(AM)A\subseteq A$. A subgroupoid $A$ of $M$ is called interior ideal of $M$
if $(MA)M\subseteq A$. A non-empty subset $A$ of a right modular groupoid $M$
is called quasi-ideal of $M$ if $QM\cap MQ\subseteq Q$. Every one sided
ideal is quasi ideal, every quasi ideal is, every bi-ideal is generalized
bi-ideal but converse is not true in general. Also every two sided ideal is
interior ideal but converse is not true.

\begin{definition}
A fuzzy subset $F$ of a right modular groupoid $M$ is called a fuzzy
interior ideal of $M$ if it satisfy the following conditions,

$(i)$ $F(xy)\geq \min \{F(x),F(y)\}$ for all $x,y\in M$.

$(ii)$ $F((xa)y)\geq F(a)$ for all $x,a,y\in M$.
\end{definition}

\begin{definition}
For a fuzzy set $F$ of a right modular groupoid $M$ and $t\in (0,1]$, the
crisp set $U(F;t)=\{x\in M$ such that $F(x)\geq t\}$ is called level subset
of $F$.
\end{definition}

\begin{definition}
A fuzzy subset $F$ of a right modular groupoid $M$ of the form
\end{definition}

\begin{equation*}
F(y)=\left \{ 
\begin{array}{c}
t\in (0,1]\text{ if }y=x \\ 
0\text{ \  \  \  \  \  \  \  \  \  \  \  \  \ if }y\neq x%
\end{array}%
\right.
\end{equation*}%
is said to be a fuzzy point with support $x$ and value $t$ and is denoted by 
$x_{t}$.

A fuzzy point $x_{t}$ is said to \textit{belong} to $($\textit{resp.
quasi-coincident with}$)$ a fuzzy set $F$, written as $x_{t}\in F$ $($%
\textit{resp. }$x_{t}qF)$ if $F(x)\geq t$ $($\textit{resp. }$F(x)+t>1)$. If $%
x_{t}\in F$ or $($\textit{resp. and}$)$ $x_{t}qF$, then we write $x_{t}\in
\vee q$ $(\in \wedge q)F$. The symbol $\overline{\in \vee q}$ means $\in
\vee q$ does not hold.

\begin{lemma}
(cf. [7]) A fuzzy set $F$ of a right modular groupoid $M$ is a fuzzy
interior ideal of $M$ if and only if $U(F;t)\left( \neq \emptyset \right) $
is an interior ideal of $M$.
\end{lemma}

\begin{definition}
\label{def of (in, in or q)F}A fuzzy set $F$ of a right modular groupoid $M$
is called an $(\in ,\in \vee q)$-fuzzy interior ideal of $M$ if for all $%
t,r\in (0,1]$ and $x,a,y\in M$.

$(A1)$ $x_{t}\in F$ and $y_{r}\in F$ implies that $(xy)_{\min \{t,r\}}\in
\vee qF.$

$(A2)$ $a_{t}\in F$ implies $((xa)y)_{t}\in \vee qF$
\end{definition}

\begin{definition}
A fuzzy set $F$ of a right modular groupoid $M$ is called an $(\in ,\in \vee
q)$-fuzzy bi-ideal of $M$ if for all $t,r\in (0,1]$ and $x,y,z\in M$.

$(B1)$ $x_{t}\in F$ and $y_{r}\in F$ implies that $(xy)_{\min \{t,r\}}\in
\vee qF.$

$(B2)$ $x_{t}\in F$ and $z_{r}\in F$ implies $((xy)z)_{\min \{t,r\}}\in \vee
qF$.
\end{definition}

\begin{lemma}
A fuzzy set $F$ of a right modular groupoid $M$ is an $(\in ,\in \vee q)$%
-fuzzy interior ideal of $M$ if and only if $U(F;t)\left( \neq \emptyset
\right) $ is an interior ideal of $M$, for all $t\in (0,0\cdot 5]$.
\end{lemma}

\begin{proof}
Let $F$ be an $(\in ,\in \vee q)$-fuzzy interior ideal of $M$. Let $x,y\in
U(F;t)$ and $t\in (0,0\cdot 5]$, then $F(x)\geq t$ and $F(y)\geq t$, so $%
F(x)\wedge F(y)\geq t$. As $F$ is an $(\in ,\in \vee q)$-fuzzy interior
ideal of $M$, so%
\begin{equation*}
F(xy)\geq F(x)\wedge F(y)\wedge 0\cdot 5\geq t\wedge 0\cdot 5=t\text{.}
\end{equation*}

Therefore, $xy\in U(F;t)$. Now if $x,y\in M$ and $a\in U(F;t)$ then $%
F(a)\geq t$ then $F((xa)y)\geq F(a)\geq t$. Therefore $((xa)y)\in U(F;t)$
and $U(F;t)$ is an interior ideal.

Conversely assume that $U(F;t)$ is a fuzzy interior ideal of $M$. If $x,y\in
U(F;t)$ then $F(x)\geq t$ and $F(y)\geq r$ which shows $x_{t}\in F$ and $%
y_{r}\in F$ as $U(F;t)$ is an interior ideal so $xy\in U(F;t)$ therefore $%
F(xy)\geq \min \{t,r\}$ implies that $(xy)_{\min \{t,r\}}\in F$, so $%
(xy)_{\min \{t,r\}}\in \vee qF$. Again let $x,y\in M$ and $a\in U(F;t)$ then 
$F(a)\geq t$ implies that $a_{t}\in F$ and $U(F;t)$ is an interior ideal so $%
((xa)y)\in U(F;t)$ then $F((xa)y)\geq t$ implies that $((xa)y)_{t}\in F$ so $%
((xa)y)_{t}\in \vee qF$. Therefore $F$ is an $(\in ,\in \vee q)$-fuzzy
interior ideal.
\end{proof}

\begin{theorem}
\label{th main}(cf. [7]) For a fuzzy set $F$ of a right modular groupoid $M$%
. The conditions $(A1)$ and $(A2)$ of Definition $4$, are equivalent to the
following,

$(A3)$ $(\forall x,y\in M)F(xy)\geq \min \{F(x),F(y),0\cdot 5\}$

$(A4)$ $F((xa)y)\geq \min \{F(a),0\cdot 5\}$.
\end{theorem}

\begin{theorem}
For a fuzzy set $F$ of a right modular groupoid $M$. The conditions $(B1)$
and $(B2)$ of Definition $5$, are equivalent to the following,

$(B3)$ $(\forall x,y\in M)F(xy)\geq \min \{F(x),F(y),0\cdot 5\}$

$(B4)$ $(\forall x,y,z\in M)F((xy)z)\geq \min \{F(x),F(y),0\cdot 5\}$.
\end{theorem}

\begin{proof}
It is similar to proof of theorem \ref{th main}.
\end{proof}

\begin{definition}
A fuzzy subset $F$ of a right modular groupoid $M$ is called an $\left( \in
,\in \vee q\right) $-fuzzy $(1,2)$ ideal of $M$ if

$(i)$ $F(xy)\geq \min \{F(x),F(y),0.5\},$

$(ii)$ $F((xa)(yz))\geq \min \{F(x),F(y),F(z),0.5\},$ for all $x,a,y,z\in M.$
\end{definition}

\begin{theorem}
Every $\left( \in ,\in \vee q\right) $-fuzzy bi-ideal is an $\left( \in ,\in
\vee q\right) $-fuzzy $(1,2)$ ideal of a right modular groupoid $M$, with
left identity.
\end{theorem}

\begin{proof}
Let $F$ be an $\left( \in ,\in \vee q\right) $-fuzzy bi-ideal of $M$ and let 
$x,a,y,z\in M$ then by using $(4)$ and $(1)$, we have%
\begin{eqnarray*}
F((xa)(yz)) &=&F(y((xa)z))\geq \min \left \{ F(y),F((xa)z),0.5\right \} \\
&=&\min \left \{ F(y),F((za)x),0.5\right \} \geq \min \left \{
(Fy),F(z),F(x),0.5,0.5\right \} \\
&=&\min \left \{ (Fy),F(z),F(x),0.5\right \} .
\end{eqnarray*}

Therefore $F$ is an $\left( \in ,\in \vee q\right) $-fuzzy $(1,2)$ ideal of
a right modular groupoid $M$.
\end{proof}

\begin{theorem}
Every $\left( \in ,\in \vee q\right) $-fuzzy interior ideal is an $\left(
\in ,\in \vee q\right) $-fuzzy $(1,2)$ ideal of a right modular groupoid $M$%
, with left identity $e$.
\end{theorem}

\begin{proof}
Let $F$ be an $\left( \in ,\in \vee q\right) $-fuzzy interior ideal of $M$
and let $x,a,y,z\in M$ then by using $(1)$, we have%
\begin{eqnarray*}
F((xa)(yz)) &\geq &\min \left \{ F(xa),F(yz),0.5\right \} \geq \min \left \{
F(xa),F(y),F(z),0.5,0.5\right \} \\
&=&\min \left \{ F((ex)a),F(y),F(z),0.5\right \} =\min \left \{
F((ax)e),F(y),F(z),0.5\right \} \\
&\geq &\min \left \{ F(x),F(y),F(z),0.5,0.5\right \} =\min \left \{
F(x),F(y),F(z),0.5\right \} .
\end{eqnarray*}

Therefore $F$ is an $\left( \in ,\in \vee q\right) -$fuzzy $(1,2)$ ideal of
a right modular groupoid $M$.
\end{proof}

\begin{theorem}
Let $\Phi :M\longrightarrow M^{%
{\acute{}}%
}$ be a homomorphism of right modular groupoids and $F$ and $G$ be $\left(
\in ,\in \vee q\right) $-fuzzy interior ideals of $M$ and $M^{%
{\acute{}}%
}$,respectively. Then

$\left( i\right) $ $\Phi ^{-1}\left( G\right) $ is an $\left( \in ,\in \vee
q\right) $-fuzzy interior ideal of $M$.

$\left( ii\right) $ If for any subset $X$ of $M$ there exist $x_{\circ }\in
X $ such that $F\left( x_{\circ }\right) =\dbigvee \left \{ F\left( x\right)
\mid x\in X\right \} $, then $\Phi \left( F\right) $ is an $\left( \in ,\in
\vee q\right) $-fuzzy interior ideal of $M%
{\acute{}}%
$ when $\Phi $ is onto.
\end{theorem}

\begin{proof}
It is same as in [4].
\end{proof}

\bigskip

{\huge Completely Regular Right Modular Groupoids}

\bigskip

\begin{definition}
A right modular groupoid $M$ is called regular, if for each $a\in M$ there
exist $x\in M$ such that $a=\left( ax\right) a$.
\end{definition}

\begin{definition}
A right modular groupoid $M$ is called left (right) regular, if for each $%
a\in M$ there exist $z\in M$ $(y\in M)$ such that $a=za^{2}$ $\left(
a=a^{2}y\right) $.
\end{definition}

\begin{definition}
A right modular groupoid $M$ is called completely regular if it is regular,
left regular and right regular.
\end{definition}

\begin{example}
Let $M=\{1,2,3,4\}$ and the binary operation $"\circ "$ defined on $M$ as
follows:%
\begin{equation*}
\begin{tabular}{c|cccc}
$\circ $ & $1$ & $2$ & $3$ & $4$ \\ \hline
$1$ & $4$ & $1$ & $2$ & $3$ \\ 
$2$ & $3$ & $4$ & $1$ & $2$ \\ 
$3$ & $2$ & $3$ & $4$ & $1$ \\ 
$4$ & $1$ & $2$ & $3$ & $4$%
\end{tabular}%
\end{equation*}%
Then clearly $(M,\circ )$ is a completely regular right modular groupoid
with left identity $4$.
\end{example}

\begin{theorem}
If $M$ is a right modular groupoid with left identity, then it is completely
regular if and only if $a\in (a^{2}M)a^{2}$.
\end{theorem}

\begin{proof}
Let $M$ be a completely regular right modular groupoid with left identity,
then for each $a\in M$ there exist $x,y,z\in M$ such that $a=(ax)a,$ $%
a=a^{2}y$ and $a=za^{2}$, so by using $(1),(4)$ and $(3)$, we get%
\begin{eqnarray*}
a &=&(ax)a=((a^{2}y)x)(za^{2})=((xy)a^{2})(za^{2})=((za^{2})a^{2})(xy) \\
&=&((a^{2}a^{2})z)(xy)=((xy)z)(a^{2}a^{2})=a^{2}(((xy)z)a^{2}) \\
&=&(ea^{2})(((xy)z)a^{2})=(a^{2}((xy)z))(a^{2}e)=(a^{2}((xy)z))((aa)e) \\
&=&(a^{2}((xy)z))((ea)a)=(a^{2}((xy)z))a^{2}\in (a^{2}M)a^{2}\text{.}
\end{eqnarray*}

Conversely, assume that $a\in (a^{2}M)a^{2}$ then clearly $a=a^{2}y$ and $%
a=za^{2}$, now using $(3),(1)$ and $(4)$, we get%
\begin{eqnarray*}
a &\in &(a^{2}M)a^{2}=\left( a^{2}M\right) \left( aa\right) =\left(
aa\right) \left( Ma^{2}\right) =\left( aa\right) \left( M\left( aa\right)
\right) \\
&=&\left( aa\right) \left( \left( eM\right) )(aa)\right) =\left( aa\right)
\left( \left( aa\right) \left( Me\right) \right) \subseteq \left( aa\right)
\left( \left( aa\right) M\right) \\
&=&(aa)\left( a^{2}M\right) =\left( (a^{2}M\right) a)a=\left( \left( \left(
aa\right) M\right) a\right) a=\left( \left( aM\right) (aa)\right) a \\
&=&\left( a\left( (Ma\right) a)\right) a\subseteq (aM)a.
\end{eqnarray*}

Therefore $M$ is completely regular.
\end{proof}

\begin{theorem}
If $M$ is a completely regular right modular groupoid, then every $\left(
\in ,\in \vee q\right) $-fuzzy $(1,2)$ ideal of $M$ is an $\left( \in ,\in
\vee q\right) $-fuzzy bi-ideal of $M$.
\end{theorem}

\begin{proof}
Let $M$ be a completely regular and $F$ is an $\left( \in ,\in \vee q\right) 
$-fuzzy $(1,2)$ ideal of $M.$ Then for $x\in M$ there exist $b\in M$ such
that $x=(x^{2}b)x^{2}$, so by using $(1)$ and $(4)$, we have

\begin{eqnarray*}
F((xa)y) &=&F((((x^{2}b)x^{2})a)y)=F((ya)((x^{2}b)x^{2})) \\
&\geq &\min \left \{ F(y),F(x^{2}b),F(x^{2}),0.5\right \} \\
&\geq &\min \left \{ F(y),F(x^{2}b),F(x),F(x),0.5,0.5\right \} \\
&=&\min \left \{ F(y),F(x^{2}b),F(x),0.5\right \} \\
&=&\min \left \{ F((xx)b),F(x),F(y),0.5\right \} \\
&=&\min \left \{ F((bx)x),F(x),F(y),0.5\right \} \\
&\geq &\min \left \{ F(bx),F(x),0.5,F(x),F(y),0.5\right \} \\
&=&\min \left \{ F(bx).F(x),F(y),0.5\right \} \\
&=&\min \left \{ F(b((x^{2}b)x^{2})),F(x),F(y),0.5\right \} \\
&=&\min \{F((x^{2}b)(bx^{2})),F(x),F(y),0.5\} \\
&=&\min \left \{ F(((bx^{2})b)x^{2}),F(x),F(y),0.5\right \} \\
&=&\min \left \{ F((((eb)x^{2})b)(xx)),F(x),F(y),0.5\right \} \\
&=&\min \left \{ F((((x^{2}b)e)b)(xx)),F(x),F(y),0.5\right \} \\
&=&\min \left \{ F(((be)(x^{2}b))(xx)),F(x),F(y),0.5\right \} \\
&=&\min \left \{ F((x^{2}((be)b))(xx)),F(x),F(y),0.5\right \} \\
&\geq &\min \left \{ F(x^{2}),F(x),F(x),0.5,F(x),F(y),0.5\right \} \\
&\geq &\min \left \{ F(x),F(x),0.5,F(x),F(x),F(x),F(y),0.5\right \} \\
&=&\min \left \{ F(x),F(y),0.5\right \} \text{.}
\end{eqnarray*}

Therefore, $F$ is an $\left( \in ,\in \vee q\right) $-fuzzy bi-ideal of $M$.
\end{proof}

\begin{theorem}
If $M$ is a completely regular right modular groupoid, then every $\left(
\in ,\in \vee q\right) $-fuzzy $(1,2)$ ideal of $M$ is an $\left( \in ,\in
\vee q\right) $-fuzzy interior ideal of $M$.
\end{theorem}

\begin{proof}
Let $M$ be a completely regular and $F$ is an $\left( \in ,\in \vee q\right) 
$-fuzzy $(1,2)$ ideal of $M.$Then for $x\in M$ there exist $y\in M$ such
that $x=(x^{2}y)x^{2}$, so by using $(4),(1),(2)$ and $(3)$, we have%
\begin{eqnarray*}
F((ax)b) &=&F((a((x^{2}y)x^{2}))b)=F(((x^{2}y)(ax^{2}))b) \\
&=&F((b(ax^{2}))(x^{2}y))=F((b(a(xx)))(x^{2}y)) \\
&=&F((b(x(ax)))(x^{2}y))=F((x(b(ax)))((xx)y)) \\
&=&F((x(b(ax)))((yx)x))=F((x(yx))((b(ax))x)) \\
&=&F(((ex)(yx))((x(ax))b))=F(((xx)(ye))((a(xx))b)) \\
&=&F(((xx)(ye))((b(xx))a)=F(((xx)(b(xx)))((ye)a)) \\
&=&F((b((xx)(xx)))((ye)a))=F((b(ye))(((xx)(xx))a)) \\
&=&F((a((xx)(xx)))((ye)b))=F(((xx)(a(xx)))((ye)b)) \\
&=&F((((ye)b)(a(xx)))(xx))=F((((ye)b)(x(ax)))(xx)) \\
&=&F((x(((ye)b)(ax)))(xx))=F((xc)(xx)) \\
&\geq &\min \{F(x),F(x),F(x),0.5\}=\min \{F(x),0.5\} \text{.}
\end{eqnarray*}

Therefore $F$ is an $\left( \in ,\in \vee q\right) $-fuzzy interior ideal of 
$M$.
\end{proof}

\begin{theorem}
Let $F$ be an$\left( \in ,\in \vee q\right) -$fuzzy bi-ideal of a right
modular groupoid $M$. If $M$ is a completely regular and $F(a)<0.5$ for all $%
x\in M$ then $F(a)=F(a^{2})$ for all $a\in M$.
\end{theorem}

\begin{proof}
Let $a\in M$ then there exist $x\in M$ such that $a=(a^{2}x)a^{2}$, then we
have

\begin{eqnarray*}
F(a) &=&F((a^{2}x)a^{2})\geq \min \{F(a^{2}),F(a^{2}),0.5\} \\
&=&\min \{F(a^{2}),0.5\}=F(a^{2})=F(aa) \\
&\geq &\min \{F(a),F(a),0.5\}=F(a)\text{.}
\end{eqnarray*}

Therefore $F(a)=F(a^{2})$.
\end{proof}

\begin{theorem}
Let $F$ be an$\left( \in ,\in \vee q\right) $-fuzzy interior ideal of a
right modular groupoid $M$. If $M$ is a completely regular and $F(a)<0.5$
for all $x\in M$ then $F(a)=F(a^{2})$ for all $a\in M$.
\end{theorem}

\begin{proof}
Let $a\in M$ then there exist $x\in M$ such that $a=(a^{2}x)a^{2}$, using $%
(4),(1)$ and $(3)$, we have%
\begin{eqnarray*}
F(a) &=&F((a^{2}x)a^{2})=F((a^{2}x)(aa))=F(a((a^{2}x)a)) \\
&=&F(a((ax)a^{2}))=F((ea)((ax)a^{2}))=F((((ax)a^{2})a)e) \\
&=&F(((aa^{2})(ax))e)=F(((xa)(a^{2}a))e)=F((((a^{2}a)a)x)e) \\
&=&F((((aa)a^{2})x)e)=F(((xa^{2})(aa))e)=F(((xa^{2})a^{2})e) \\
&\geq &\min \{F(a^{2}),0.5\}=F(a^{2})=F(aa) \\
&\geq &\min \{F(a),F(a),0.5\} \geq \min \{F(a),0.5\}=F(a)\text{.}
\end{eqnarray*}

Therefore $F(a)=F(a^{2})$.
\end{proof}

$\bigskip $

${\huge (\in ,\in \vee q}_{k}{\huge )}${\huge -fuzzy Ideals in Right Modular
Groupoids}

\bigskip

It has been given in $[3]$ that $x_{t}q_{k}F$ is the generalizations of $%
x_{t}qF$, where $k$ is an arbitrary element of $[0,1)$ as $x_{t}q_{k}F$ if $%
F(x)+t+k>1$. If $x_{t}\in F$ or $x_{t}qF$ implies $x_{t}\in q_{k}F$. Here we
discuss the behavior of $(\in ,\in \vee q_{k})$-fuzzy left ideal, $(\in ,\in
\vee q_{k})$-fuzzy right ideal,$(\in ,\in \vee q_{k})$-fuzzy interior ideal, 
$(\in ,\in \vee q_{k})$-fuzzy bi-ideal, $(\in ,\in \vee q_{k})$-fuzzy
quasi-ideal in the completely regular right modular groupoid $M$.

\begin{definition}
A fuzzy subset $F$ of a right modular groupoid $M$ is called an $(\in ,\in
\vee q_{k})$-fuzzy subgroupoid of $M$ if for all $x,y\in M$ and $t,r\in
(0,1] $ the following condition holds%
\begin{equation*}
x_{t}\in F,y_{r}\in F\text{ implies }(xy)_{\min \{t,r\}}\in \vee q_{k}F\text{%
.}
\end{equation*}
\end{definition}

\begin{theorem}
Let $F$ be a fuzzy subset of $M$. Then $F$ is an $(\in ,\in \vee q_{k})$%
-fuzzy subgroupoid of $M$ if and only if $F(xy)\geq \min \{F(x),F(y),\frac{%
1-k}{2}\}$.
\end{theorem}

\begin{proof}
It is similar to the proof of theorem \ref{th main}.
\end{proof}

\begin{definition}
A fuzzy subset $F$ of a right modular groupoid $M$ is called an $(\in ,\in
\vee q_{k})$-fuzzy left (right) ideal of $M$ if for all $x,y\in M$ and $%
t,r\in (0,1]$ the following condition holds%
\begin{equation*}
y_{t}\in F\text{ implies }\left( xy\right) _{t}\in \vee q_{k}F\text{ }\left(
y_{t}\in F\text{ implies }\left( yx\right) _{t}\in \vee q_{k}F\right) \text{.%
}
\end{equation*}
\end{definition}

\begin{theorem}
Let $F$ be a fuzzy subset of $M$. Then $F$ is an $(\in ,\in \vee q_{k})$%
-fuzzy left (right) ideal of $M$ if and only if $F(xy)\geq \min \{F(y),\frac{%
1-k}{2}\} \left( F(xy)\geq \min \{F(x),\frac{1-k}{2}\} \right) $.
\end{theorem}

\begin{proof}
Let $F$ be an $(\in ,\in \vee q_{k})$-fuzzy left ideal of $M$. Suppose that
there exist $x,y\in M$ such that $F\left( xy\right) <\min \left \{ F\left(
y\right) ,\frac{1-k}{2}\right \} $. Choose a $t\in (0,1]$ such that $F\left(
xy\right) <t<\min \left \{ F\left( y\right) ,\frac{1-k}{2}\right \} $. Then $%
y_{t}\in F$ but $\left( xy\right) _{t}\notin F$ and $F\left( xy\right) +t+k<%
\frac{1-k}{2}+\frac{1-k}{2}+k=1$, so $\left( xy\right) _{t}\overline{\in
\vee q_{k}}F$, a contradiction. Therefore $F(xy)\geq \min \{F(y),\frac{1-k}{2%
}\}$.

Conversely, assume that $F(xy)\geq \min \{F(y),\frac{1-k}{2}\}$. Let $x,y\in
M$ and $t\in (0,1]$ such that $y_{t}\in M$ then $F(y)\geq t.$ then $%
F(xy)\geq \min \{F(y),\frac{1-k}{2}\} \geq \min \{t,\frac{1-k}{2}\}$. If $t>%
\frac{1-k}{2}$ then $F(xy)\geq \frac{1-k}{2}$. So $F\left( xy\right) +t+k>%
\frac{1-k}{2}+\frac{1-k}{2}+k=1$, which implies that $\left( xy\right)
_{t}q_{k}F$. If $t\leq \frac{1-k}{2}$, then $F(xy)\geq t$. Therefore $%
F(xy)\geq t$ which implies that $\left( xy\right) _{t}\in F$. Thus $\left(
xy\right) _{t}\in \vee q_{k}F$.
\end{proof}

\begin{corollary}
A fuzzy subset $F$ of a right modular groupoid $M$ is called an $(\in ,\in
\vee q_{k})$-fuzzy ideal of $M$ if and only if $F(xy)\geq \min \{F(y),\frac{%
1-k}{2}\}$ and $F(xy)\geq \min \{F(x),\frac{1-k}{2}\}$.
\end{corollary}

\begin{definition}
A fuzzy subset $F$ of a right modular groupoid $M$ is called an $(\in ,\in
\vee q_{k})$-fuzzy bi-ideal of $M$ if for all $x,y,z\in M$ and $t,r\in (0,1]$
the following conditions hold

$\left( i\right) $ If $x_{t}\in F$ and $y_{r}\in M$ implies $\left(
xy\right) _{\min \left \{ t,r\right \} }\in \vee q_{k}F$,

$\left( ii\right) $ If $x_{t}\in F$ and $z_{r}\in M$ implies $\left( \left(
xy\right) z\right) _{\min \left \{ t,r\right \} }\in \vee q_{k}F$.
\end{definition}

\begin{theorem}
Let $F$ be a fuzzy subset of $M$. Then $F$ is an $(\in ,\in \vee q_{k})$%
-fuzzy bi-ideal of $M$ if and only if

$\left( i\right) F(xy)\geq \min \{F\left( x\right) ,F(y),\frac{1-k}{2}\}$
for all $x,y\in M$ and $k\in \lbrack 0,1)$,

$\left( ii\right) F((xy)z)\geq \min \{F(x),F\left( z\right) ,\frac{1-k}{2}\}$
for all $x,y,z\in M$ and $k\in \lbrack 0,1)$.
\end{theorem}

\begin{proof}
It is similar to the proof of theorem \ref{th main}.
\end{proof}

\begin{corollary}
Let $F$ be a fuzzy subset of $M$. Then $F$ is an $(\in ,\in \vee q_{k})$%
-fuzzy generalized bi-ideal of $M$ if and only if $F((xy)z)\geq \min
\{F(x),F\left( z\right) ,\frac{1-k}{2}\}$ for all $x,y,z\in M$ and $k\in
\lbrack 0,1)$.
\end{corollary}

\begin{definition}
A fuzzy subset $F$ of a right modular groupoid $M$ is called an $(\in ,\in
\vee q_{k})$-fuzzy interior ideal of $M$ if for all $x,,a,y\in M$ and $%
t,r\in (0,1]$ the following conditions hold

$\left( i\right) $ If $x_{t}\in F$ and $y_{r}\in M$ implies $\left(
xy\right) _{\min \left \{ t,r\right \} }\in \vee q_{k}F$,

$\left( ii\right) $ If $a_{t}\in M$ implies $\left( \left( xa\right)
y\right) _{\min \left \{ t,r\right \} }\in \vee q_{k}F$.
\end{definition}

\begin{theorem}
Let $F$ be a fuzzy subset of $M$. Then $F$ is an $(\in ,\in \vee q_{k})$%
-fuzzy interior ideal of $M$ if and only if

$\left( i\right) $ $F(xy)\geq \min \{F\left( x\right) ,F(y),\frac{1-k}{2}\}$
for all $x,y\in M$ and $k\in \lbrack 0,1)$,

$\left( ii\right) $ $F((xa)y)\geq \min \{F(a),\frac{1-k}{2}\}$ for all $%
x,a,y\in M$ and $k\in \lbrack 0,1)$.
\end{theorem}

\begin{proof}
It is similar to the proof of theorem \ref{th main}.
\end{proof}

\begin{lemma}
The intersection of any family of $(\in ,\in \vee q_{k})$-fuzzy interior
ideals of right modular groupoid $M$ is an $(\in ,\in \vee q_{k})$-fuzzy
interior ideal of $M$.
\end{lemma}

\begin{proof}
Let $\left \{ F_{i}\right \} _{i\in I}$ be a family of $(\in ,\in \vee q_{k})$%
-fuzzy interior ideals of $M$ and $x,a,y\in M$. Then $(\wedge _{i\in
I}F_{i})(\left( xa)y\right) =\wedge _{i\in I}(F_{i}(\left( xa\right) y)$. As
each $F_{i}$ is an $(\in ,\in \vee q_{k})$-fuzzy interior ideal of $M$, so $%
F_{i}((xa)y)\geq F_{i}(a)\wedge \frac{1-k}{2}$ for all $i\in I$. Thus 
\begin{eqnarray*}
\left( \wedge _{i\in I}F_{i}\right) (\left( xa)y\right) &=&\wedge _{i\in
I}(F_{i}(\left( xa\right) y)\geq \wedge _{i\in I}\left( F_{i}(a)\wedge \frac{%
1-k}{2}\right) \\
&=&\left( \wedge _{i\in I}F_{i}(a)\right) \wedge \frac{1-k}{2}=\left( \wedge
_{i\in I}F_{i}\right) \left( a\right) \wedge \frac{1-k}{2}\text{.}
\end{eqnarray*}

Therefore $\wedge _{i\in I}F_{i}$ is an $(\in ,\in \vee q_{k})$-fuzzy
interior ideal of $M$.
\end{proof}

\begin{definition}
A fuzzy subset $F$ of a right modular groupoid $M$ is called an $(\in ,\in
\vee q_{k})$-fuzzy quasi-ideal of $M$ if following condition holds%
\begin{equation*}
F\left( x\right) \geq \min \left \{ (F\circ 1)\left( x\right) ,\left( 1\circ
f\right) (x),\frac{1-k}{2}\right \} \text{.}
\end{equation*}

where $1$ is the fuzzy subset of $M$ mapping every element of $M$ on $1$.
\end{definition}

\begin{lemma}
If $M$ is a completely regular right modular groupoid with left identity,
then a fuzzy subset $F$ is an $(\in ,\in \vee q_{k})$-fuzzy right ideal of $%
M $ if and only if $F$ is an $(\in ,\in \vee q_{k})$-fuzzy left ideal of $M$.
\end{lemma}

\begin{proof}
Let $F$ be an $(\in ,\in \vee q_{k})$-fuzzy right ideal of a completely
regular right modular groupoid $M$, then for each $a\in M$ there exist $x\in
M$ such that $a=(a^{2}x)a^{2}$, then by using $(1)$, we have%
\begin{eqnarray*}
F(ab) &=&F(((a^{2}x)a^{2})b)=F((ba^{2})(a^{2}x)) \\
&\geq &F(ba^{2})\wedge \frac{1-k}{2}\geq F(b)\wedge \frac{1-k}{2}\text{.}
\end{eqnarray*}

Conversely, assume that $F$ is an $(\in ,\in \vee q_{k})$-fuzzy right ideal
of $M$, then by using $\left( 1\right) $, we have%
\begin{eqnarray*}
F(ab) &=&F(((a^{2}x)a^{2})b)=F((ba^{2})(a^{2}x)) \\
&\geq &F(a^{2}x)\wedge \frac{1-k}{2}=F((aa)x)\wedge \frac{1-k}{2} \\
&=&F((xa)a)\wedge \frac{1-k}{2}\geq F(a)\wedge \frac{1-k}{2}\text{.}
\end{eqnarray*}
\end{proof}

\begin{theorem}
If $M$ is a completely regular right modular groupoid with left identity,
then a fuzzy subset $F$ is an $(\in ,\in \vee q_{k})$-fuzzy ideal of $M$ if
and only if $F$ is an $(\in ,\in \vee q_{k})$-fuzzy interior ideal of $M$.
\end{theorem}

\begin{proof}
Let $F$ be an $(\in ,\in \vee q_{k})$-fuzzy interior ideal of a completely
regular right modular groupoid $M$, then for each $a\in M$ there exist $x\in
M$ such that $a=(a^{2}x)a^{2}$, then by using $\left( 4\right) $ and $\left(
1\right) $, we have%
\begin{equation*}
F(ab)=F(((a^{2}x)a^{2})b)\geq F(aa)\geq F(a)\wedge F(a)\wedge \frac{1-k}{2}%
\text{, and}
\end{equation*}%
\begin{eqnarray*}
F(ab) &=&F(a((b^{2}y)b^{2}))F((b^{2}y)(ab^{2}))=F(((bb)y)(ab^{2})) \\
&=&F(((yb)b)(ab^{2}))\geq F(b)\wedge \frac{1-k}{2}\text{.}
\end{eqnarray*}

The converse is obvious.
\end{proof}

\begin{theorem}
If $M$ is a completely regular right modular groupoid with left identity,
then a fuzzy subset $F$ is an $(\in ,\in \vee q_{k})$-fuzzy generalized
bi-ideal of $M$ if and only if $F$ is an $(\in ,\in \vee q_{k})$-fuzzy
bi-ideal of $M$.
\end{theorem}

\begin{proof}
Let $F$ be an $(\in ,\in \vee q_{k})$-fuzzy generalized bi-ideal of a
completely regular right modular groupoid $M$, then for each $a\in M$ there
exist $x\in M$ such that $a=(a^{2}x)a^{2}$, then by using $(4)$, we have%
\begin{eqnarray*}
F(ab) &=&F(((a^{2}x)a^{2})b)=F(((a^{2}x)(aa))b) \\
&=&F((a((a^{2}x)a))b)\geq F(a)\wedge F(b)\wedge \frac{1-k}{2}\text{.}
\end{eqnarray*}

The converse is obvious.
\end{proof}

\begin{theorem}
If $M$ is a completely regular right modular groupoid with left identity,
then a fuzzy subset $F$ is an $(\in ,\in \vee q_{k})$-fuzzy bi-ideal of $M$
if and only if $\ F$ is an $(\in ,\in \vee q_{k})$-fuzzy two sided ideal of $%
M$.
\end{theorem}

\begin{proof}
Let $F$ be an $(\in ,\in \vee q_{k})$-fuzzy bi-ideal of a completely regular
right modular groupoid $M$, then for each $a\in M$ there exist $x\in M$ such
that $a=(a^{2}x)a^{2}$, then by using $(1)$ and $\left( 4\right) $, we have%
\begin{eqnarray*}
F(ab) &=&F(((a^{2}x)a^{2})b)=F((((aa)x)a^{2})b)=F((((xa)a)a^{2})b) \\
&=&F((ba^{2})((xa)a))=F((b(aa))((aa)x))=F((a(ba))((aa)x)) \\
&=&F((aa((a(ba))x))=F((((a(ba))x)a)a)=F((((b(aa))x)a)a) \\
&=&F((((x(aa))b)a)a)=F((((a(xa))b)a)a)=F(((ab)(a(xa)))a) \\
&=&F((a((ab)(xa)))a)\geq F(a)\wedge F(a)\wedge \frac{1-k}{2}=F(a)\wedge 
\frac{1-k}{2}\text{.}
\end{eqnarray*}

And, by using $\left( 4\right) ,\left( 1\right) $ and $\left( 3\right) $,we
have%
\begin{eqnarray*}
F(ab) &=&F(a((b^{2}y)b^{2}))=F((b^{2}y)(ab^{2}))=F(((bb)y)(a(bb))) \\
&=&F(((a(bb)y)(bb))=F(((a(bb))(ey))(bb))=F(((ye)((bb)a))(bb)) \\
&=&F((bb)((ye)a))(bb))\geq F(bb)\wedge \frac{1-k}{2}\text{.}\geq F(b)\wedge
F(b)\wedge \frac{1-k}{2}\text{.} \\
&=&F(b)\wedge \frac{1-k}{2}\text{.}
\end{eqnarray*}

The converse is obvious.
\end{proof}

\begin{theorem}
If $M$ is a completely regular right modular groupoid with left identity,
then a fuzzy subset $F$ is an $(\in ,\in \vee q_{k})$-fuzzy quasi-ideal of $%
M $ if and only if $\ F$ is an $(\in ,\in \vee q_{k})$-fuzzy two sided ideal
of $M$.
\end{theorem}

\begin{proof}
Let $F$ be an $(\in ,\in \vee q_{k})$-fuzzy quasi-ideal of a completely
regular right modular groupoid $M$, then for each $a\in M$ there exist $x\in
M$ such that $a=(a^{2}x)a^{2}$, then by using $(1)$, $\left( 3\right) $ and $%
\left( 4\right) $, we have%
\begin{eqnarray*}
ab &=&((a^{2}x)a^{2})b=(ba^{2})\left( a^{2}x\right) =\left( xa^{2}\right)
\left( a^{2}b\right) \\
&=&\left( x\left( aa\right) \right) \left( a^{2}b\right) =\left( a\left(
xa\right) \right) \left( a^{2}b\right) =\left( \left( a^{2}b\right) \left(
xa\right) \right) a\text{.}
\end{eqnarray*}

Then%
\begin{eqnarray*}
F(ab) &\geq &(F\circ 1)\left( ab\right) \wedge \left( 1\circ F\right) \left(
ab\right) \wedge \frac{1-k}{2} \\
&=&\dbigvee \limits_{ab=pq}\left \{ F(p)\wedge 1\left( q\right) \right \}
\wedge (1\circ F)(ab)\wedge \frac{1-k}{2} \\
&\geq &F\left( a\right) \wedge 1\left( b\right) \wedge
\dbigvee \limits_{ab=lm}\left \{ F\left( l\right) \wedge 1(m)\right \} \wedge 
\frac{1-k}{2} \\
&=&F\left( a\right) \wedge \dbigvee \limits_{ab=\left( \left( a^{2}b\right)
\left( xa\right) \right) a}\left \{ 1\left( \left( a^{2}b\right) \left(
xa\right) \right) \wedge F\left( a\right) \right \} \\
&\geq &F\left( a\right) \wedge 1\left( \left( a^{2}b\right) \left( xa\right)
\right) \wedge F\left( a\right) \wedge \frac{1-k}{2}=F\left( a\right) \wedge 
\frac{1-k}{2}\text{.}
\end{eqnarray*}

Also by using $\left( 4\right) $ and $\left( 1\right) $, we have%
\begin{eqnarray*}
ab &=&a\left( (b^{2}y\right) b^{2})=\left( b^{2}y\right) \left(
ab^{2}\right) =\left( \left( bb\right) y\right) \left( ab^{2}\right) \\
&=&\left( \left( ab^{2}\right) y\right) \left( bb\right) =b\left( (\left(
ab^{2}\right) y\right) b)\text{.}
\end{eqnarray*}

Then%
\begin{eqnarray*}
F(ab) &\geq &\left( F\circ 1\right) \left( ab\right) \wedge \left( 1\circ
F\right) \left( ab\right) \wedge \frac{1-k}{2} \\
&=&\dbigvee \limits_{ab=pq}\left \{ F(p)\wedge 1\left( q\right) \right \}
\wedge \dbigvee \limits_{ab=lm}\left \{ 1(l)\wedge F\left( m\right) \right \}
\wedge \frac{1-k}{2} \\
&=&\dbigvee \limits_{ab=b\left( (\left( ab^{2}\right) y\right) b)}\left \{
F(b)\wedge 1(\left( \left( ab^{2}\right) y\right) b)\right \} \wedge
\dbigvee \limits_{ab=ab}\left \{ 1(a)\wedge F(b)\right \} \wedge \frac{1-k}{2}
\\
&\geq &F\left( b\right) \wedge 1(\left( \left( ab^{2}\right) y\right)
b)\wedge 1\left( a\right) \wedge F\left( b\right) \wedge \frac{1-k}{2}%
=F\left( b\right) \wedge \frac{1-k}{2}\text{.}
\end{eqnarray*}

The converse is obvious
\end{proof}

\begin{remark}
We note that in a completely regular right modular groupoid $M$ with left
identity, $(\in ,\in \vee q_{k})$-fuzzy left ideal $(\in ,\in \vee q_{k})$%
-fuzzy right ideal, $(\in ,\in \vee q_{k})$ fuzzy ideal,$(\in ,\in \vee
q_{k})$-fuzzy interior ideal, $(\in ,\in \vee q_{k})$-fuzzy bi-ideal, $(\in
,\in \vee q_{k})$-fuzzy generalized bi-ideal and $(\in ,\in \vee q_{k})$%
-fuzzy quasi-ideal coincide with each other.
\end{remark}

\begin{theorem}
If $M$ is a completely regular right modular groupoid then $F\wedge
_{k}G=F\circ _{k}G$ for every $(\in ,\in \vee q_{k})$- fuzzy ideal $F$ and $%
G $ of $M$.
\end{theorem}

\begin{proof}
Let $F$ is an $(\in ,\in \vee q_{k})$- fuzzy right ideal of $M$ and $G$ is
an $(\in ,\in \vee q_{k})$-fuzzy left ideal of $M$, and $M$ is a completely
regular then for each $a\in M$ there exist $x\in M$ such that $%
a=(a^{2}x)a^{2}$, so we have%
\begin{eqnarray*}
(F\circ _{k}G)\left( a\right) &=&(F\circ G)\left( a\right) \wedge \frac{1-k}{%
2}=\dbigvee \limits_{a=pq}\left \{ F(p)\wedge G\left( q\right) \right \} \wedge 
\frac{1-k}{2} \\
&\geq &F(a^{2}x)\wedge G\left( a^{2}\right) \wedge \frac{1-k}{2}\geq
F(aa)\wedge G\left( aa\right) \wedge \frac{1-k}{2} \\
&\geq &F(aa)\wedge G\left( aa\right) \wedge \frac{1-k}{2}\geq F(a)\wedge
G\left( a\right) \wedge \frac{1-k}{2} \\
&=&\left( F\wedge G\right) (a)\wedge \frac{1-k}{2}=\left( F\wedge
_{k}G\right) (a)\text{.}
\end{eqnarray*}

Therefore $F\wedge _{k}G\leq F\circ _{k}G$, again%
\begin{eqnarray*}
(F\circ _{k}G)\left( a\right) &=&(F\circ G)\left( a\right) \wedge \frac{1-k}{%
2}=\left( \dbigvee \limits_{a=pq}\left \{ F(p)\wedge G\left( q\right) \right \}
\right) \wedge \frac{1-k}{2} \\
&=&\dbigvee \limits_{a=pq}\left \{ F(p)\wedge G\left( q\right) \wedge \frac{1-k%
}{2}\right \} \leq \dbigvee \limits_{a=pq}\left \{ \left( F(pq\wedge G\left(
pq\right) \right) \wedge \frac{1-k}{2}\right \} \\
&=&F(a)\wedge G\left( a\right) \wedge \frac{1-k}{2}=(F\wedge _{k}G)\left(
a\right) \text{.}
\end{eqnarray*}

Therefore $F\wedge _{k}G\geq F\circ _{k}G$. Thus $F\wedge _{k}G=F\circ _{k}G$%
.
\end{proof}

\begin{definition}
A right modular groupoid $M$ is called weakly regular if for each $a$ in $M$
there exists $x$ and $y$ in $M$ such that $a=(ax)(ay)$.
\end{definition}

It is easy to see that right regular, left regular and weakly regular
coincide in a right modular groupoid with left identity.

\begin{theorem}
For a weakly regular right modular groupoid $M$ with left identity, $%
(G\wedge _{k}F)\wedge _{k}H)\leq ((G\circ _{k}F)\circ _{k}H)$, where $G$ is
an $(\in ,\in \vee q_{k})$- fuzzy right ideal, $F$ is an $(\in ,\in \vee
q_{k})$- fuzzy interior ideal and $H$ is an $(\in ,\in \vee q_{k})$- fuzzy
left ideal.
\end{theorem}

\begin{proof}
Let $M$ be a weakly regular right modular groupoid with left identity, then
for each $a$ in $M$ there exists $x$ and $y$ in $M$ such that $a=(ax)(ay)$,
then by using $\left( 3\right) $, we get $a=(ya)(xa)$, and also by using $%
\left( 4\right) $ and $\left( 3\right) $, we get 
\begin{equation*}
ya=y((ax)(ay))=(ax)(y(ay))=(ax)((ey)(ay))=(ax)((ya)(ye))
\end{equation*}

Then%
\begin{eqnarray*}
((G\circ _{k}F)\circ _{k}H)(a) &=&\dbigvee_{a=pq}\left \{ (G\circ
_{k}F)(p)\wedge H(q)\right \} \geq (G\circ _{k}F)(ya)\wedge H(xa) \\
&\geq &(G\circ _{k}F)(ya)\wedge H(a)\wedge \frac{1-k}{2} \\
&=&\dbigvee_{ya=bc}\left \{ G(b)\wedge F(c)\right \} \wedge H(a)\wedge \frac{%
1-k}{2} \\
&\geq &(G(ax)\wedge F((ya)(ye)))\wedge H(a)\wedge \frac{1-k}{2} \\
&\geq &(G(a)\wedge \frac{1-k}{2}\wedge F(a)\wedge \frac{1-k}{2})\wedge
H(a)\wedge \frac{1-k}{2} \\
&=&(G(a)\wedge F(a))\wedge H(a)\wedge \frac{1-k}{2} \\
&=&((G\wedge F)\wedge H)(a)\wedge \frac{1-k}{2} \\
&=&((G\wedge _{k}F)\wedge _{k}H)(a)\text{.}
\end{eqnarray*}

Therefore, $(G\wedge _{k}F)$ $\wedge _{k}H)\leq ((G\circ _{k}F)\circ _{k}H)$.
\end{proof}

\begin{theorem}
For a weakly regular right modular groupoid $M$ with left identity, $%
F_{k}\leq ((F\circ _{k}1)\circ _{k}F)$, where $F$ is an $(\in ,\in \vee
q_{k})$- fuzzy interior ideal.
\end{theorem}

\begin{proof}
Let $M$ be a weakly regular right modular groupoid with left identity, then
for each $a\in M$ there exist $x,y\in M$ such that $a=(ax)(ay)$, then by
using $(1)$ $a=((ay)x)a$. Also by using $\left( 1\right) $ and $\left(
4\right) $, we have%
\begin{eqnarray*}
(ay) &=&(((ax)(ay))y)=((y(ay))(ax))=((a(yy))(ax)) \\
&=&(((ax)(yy))a)=((((yy)x)a)a)\text{.}
\end{eqnarray*}%
Then%
\begin{eqnarray*}
((F\circ _{k}1)\circ _{k}F)(a) &=&((F\circ 1)\circ F)(a)\wedge \frac{1-k}{2}
\\
&=&\dbigvee_{a=pq}\left \{ (F\circ 1)(p)\wedge F(q)\right \} \wedge \frac{1-k}{%
2} \\
&\geq &(F\circ 1)(((ay)x)\wedge F(a)\wedge \frac{1-k}{2} \\
&=&\dbigvee_{(ay)x=(bc)}\left \{ F(b)\wedge 1(c)\right \} \wedge F(a)\wedge 
\frac{1-k}{2} \\
&\geq &F(ay)\wedge 1(x)\wedge F(a)\wedge \frac{1-k}{2} \\
&=&F((((yy)x)a)a)\wedge F(a)\wedge \frac{1-k}{2} \\
&\geq &F(a)\wedge \frac{1-k}{2}\wedge F(a)\wedge \frac{1-k}{2} \\
&=&F(a)\wedge \frac{1-k}{2}=F_{k}(a)\text{.}
\end{eqnarray*}

Therefore, $F_{k}\leq ((F\circ _{k}1)\circ _{k}F)$.
\end{proof}

\end{document}